\documentclass[12pt]{amsart}
\usepackage{amssymb,amsmath}
\def\C {{\mathcal C}}

\def\H {{\mathcal H}}

\def\A {{\mathbb A}}

\def\F {{\mathsf{F}}}
\def\P {{\mathsf{P}}}

\def\E {{\mathsf{E}}}
\def\LL {{\mathsf{V}}}

\def\R {\mathbb{R}}
\def\N {\mathbb{N}}
\def\D {{\mathfrak{D}}}

\def\Re {\mathfrak{Re\,}}
\def\Im {\mathfrak{Im\,}}
\def\e{{\rm e}}
\def\d{{\rm d}}
\def\ddt{\frac{\d}{\d t}}
\def\i{{\rm i}}

\def \l {\langle}
\def \r {\rangle}
\def \and {{\qquad\text{and}\qquad}}

\newtheorem{proposition}{Proposition}[section]
\newtheorem{theorem}[proposition]{Theorem}

\newtheorem{lemma}[proposition]{Lemma}

\theoremstyle{definition}

\newtheorem{remark}[proposition]{Remark}
\newtheorem*{Acknowledgments}{Acknowledgments}
\numberwithin{equation}{section}
\def \au {\rm}
\def \ti {\it}
\def \jou {\rm}
\def \bk {\it}
\def \no#1#2#3 {{\bf #1} (#3), #2.}
\def \eds#1#2#3 {#1, #2, #3.}

\title[MGT equation with memory of type II]
{On the MGT equation with memory of type II}

\author[F. Dell'Oro, I. Lasiecka and V. Pata]
{Filippo Dell'Oro, Irena Lasiecka and Vittorino Pata}

\address{Politecnico di Milano - Dipartimento di Matematica
\newline\indent
Via Bonardi 9, 20133 Milano, Italy}
\email{filippo.delloro@polimi.it {\rm (F. Dell'Oro)}}
\email{vittorino.pata@polimi.it {\rm (V. Pata)}}

\address{University of Memphis - Department of Mathematical Sciences
\newline\indent
Memphis, TN, 38152, USA}
\email{lasiecka@memphis.edu {\rm (I. Lasiecka)}}

\begin{document}

\begin{abstract}
We consider the Moore-Gibson-Thompson equation with memory of type II
$$
\partial_{ttt} u(t) + \alpha \partial_{tt} u(t) + \beta A \partial_t u(t) + \gamma Au(t)-\int_0^t g(t-s) A \partial_t u(s)\d s=0
$$
where $A$ is a strictly positive selfadjoint linear operator (bounded or unbounded) and
$\alpha,\beta,\gamma>0$ satisfy the relation $\gamma\leq\alpha\beta$.
First, we prove a well-posedness result without requiring any restriction on the total mass $\varrho$ of $g$.
Then we show that it is always possible
to find memory kernels $g$, complying with the usual mass restriction $\varrho<\beta$, such that the equation admits solutions
with energy growing exponentially fast. In particular, this provides the answer to a question raised in \cite{DLP}.
\end{abstract}

\maketitle


\section{Introduction}

\noindent
Let $(H,\l \cdot,\cdot \r, \|\cdot\| ) $ be a separable real Hilbert space, and let
$$A: \D(A)\subset H \to H$$
be a strictly positive selfadjoint linear operator (bounded or unbounded).
We consider the Moore-Gibson-Thompson (MGT) equation with memory of type II
\begin{equation}
\label{VOLT}
\partial_{ttt} u(t) + \alpha \partial_{tt} u(t) + \beta A \partial_t u(t) + \gamma Au(t)-\int_0^t g(t-s) A \partial_t u(s)\d s=0,
\end{equation}
where $\alpha,\beta,\gamma$ are strictly positive fixed constants subject to the structural constraint
\begin{equation}
\label{S1}
\gamma \leq \alpha\beta,
\end{equation}
and the so-called memory kernel $g:[0,\infty)\to [0,\infty)$ is an absolutely continuous 
nonincreasing function of total mass
$$\varrho=\int_0^\infty g(s)\d s>0.$$

The MGT equation without memory, i.e.\
\begin{equation}
\label{nomem}
\partial_{ttt} u + \alpha \partial_{tt} u + \beta A \partial_t u + \gamma Au=0,
\end{equation}
is a model arising in acoustics and
accounting for the second sound effects and the
associated thermal relaxations in viscous fluids \cite{PJ,MG,STO,TOM}.
The case $\gamma<\alpha\beta$ is referred to as {\it subcritical}, since in this regime
the associated solution semigroup exhibits an exponential decay in the natural weak energy space
$$\H=\D(A^\frac{1}{2})\times \D(A^\frac{1}{2})\times H.$$
On the contrary,
the case $\gamma=\alpha\beta$ is
{\it critical}, since stability (even the nonuniform one) is lost \cite{KLM,TRIG}. Finally, in the {\it supercritical} case
$\gamma>\alpha\beta$, there exist trajectories whose energy blows up exponentially \cite{DPMGT,KLM,TRIG}.
If additional molecular relaxation phenomena are taken into
account, integral terms pop up in the MGT equation, leading to \eqref{VOLT} with
a nonnull memory kernel \cite{jordan1,LW,LW2,lebon,ostrovsky}. In more generality,
the convolution term appearing in \eqref{VOLT} can be taken of the form
$\int_0^t g(t-s) A w (s)\,\d s$, where the variable $w$ is of the following three types:
$$w(s) =
\begin{cases}
u(s) \quad &\text{(type I)},\\\noalign{\vskip0.5mm}
\partial_t u(s)\quad &\text{(type II)},\\\noalign{\vskip0.5mm}
k u(s) + \partial_t u(s)\quad\text{for}\,\, k>0\quad &\text{(type III)}.
\end{cases}$$
For the memory of type I, the picture is quite understood. As shown in \cite{LW},
in the subcritical case and under proper decay
assumptions on the kernel $g$ all the solutions
converge exponentially to zero. Instead, in the critical case
the decay is only strong, with a counterexample
to exponential stability if the operator $A$ is unbounded \cite{DLP}.
An interesting question becomes what is the effect of the memory of type II.
Here, in the subcritical case and with strong restrictions
on the mass of memory kernel including $\varrho \ll \beta$, one shows the
exponential decay of the energy \cite{LW}.
In the critical case, exponential stability holds, but with a very special choice of memory of type III,
namely $w = \gamma \beta^{-1} u + \partial_t u $.

It is worth noting that in all the results mentioned above a structural restriction on $\varrho$ is required.
For the case of memory of type II, such a restriction reads
$$\varrho<\beta.$$
To better understand this issue,
an interesting comparison can be made
with the MGT equation without memory \eqref{nomem}, which is shown to be ill-posed in $\H$ if
$A$ is unbounded and $\beta=0$ (but the same is true if $\beta\leq 0$),
in the sense that the equation does not
generate a strongly continuous semigroup (see \cite{KLM}). And indeed, equation \eqref{nomem}
with $(\beta-\varrho)A \partial_t u$ in place of
the term $\beta A \partial_t u$ can be considered the limiting case of \eqref{VOLT}
when the kernel
$g$ converges to a multiple $\varrho$ of the Dirac mass at $0^+$. This would somehow indicate that
some problems might arise when $\varrho\geq\beta$. Quite unexpectedly, as it will be shown in this work,
it is instead possible to have existence and uniqueness of solutions in the natural weak energy space $\H$,
no matter how is the size of $\varrho$.
In fact, the same picture occurs for the MGT equation with memory of type I or III.

A second intriguing problem is to fully understand the effects of the memory of type II on the longtime dynamics,
\emph{within} the restriction $\varrho<\beta$ (indeed, if $\varrho\geq \beta$, blow up at infinity is the general rule).
In particular, whether this damping alone is able to stabilize the equation in the critical case.
As we shall see in this paper, the answer is negative.
Even more is true:\ a {\it subcritical} MGT equation can be exponentially destabilized by ``large"
effects of the memory of type II. Hence, a posteriori,
for such an equation we may say that no critical value changing the asymptotic dynamics exists, in the sense that
blow up of solutions appears, both in the subcritical and in the critical regimes.
In particular, this provides an answer to a question raised in \cite{DLP}. Actually, in the same paper \cite{DLP}
a heuristic explanation was given, by noting that the action
of the memory of type II can be interpreted as an addition of a ``stabilizer" and of an ``antidamper" to  the MGT equation. To wit, observe that
$$- \int_0^t   g(t-s) A \partial_t u (s) \d s =
- g(0) A u(t)  + g(t) A u(0) - \int_0^t  g'(t-s) A u(s) \d s.$$
The above formula indicates that the memory of type II provides two opposite effects.
The static damping term $ g(0) A u $  moves  the original critical value
$\gamma=\alpha\beta$ to the noncritical region $ \gamma=\alpha\beta-g(0)$. On the other hand, the
viscoelastic term $-\int_0^t g'(t-s) A u(s) \d s $ acts as an ``antidamper", due to the  negative sign of $g'$.
This renders the issue quite interesting, as it is not clear which ``damping" wins the game. Of course, the value $g(0)$ and
$g'(t)$ will play a crucial role.

\subsection*{Comparison with the previous literature}
The recent paper \cite{brasilian} is concerned with the existence, uniqueness
and stability of the MGT equation with {\it infinite memory},
i.e.\ with a more general convolution term of the form
$\int_{-\infty}^t g(t-s) A w (s) \d s$. In the case of memory of type I,
Theorem 3.7 therein proves the exponential decay of solutions
in the subcritical case.
This result has been already shown for {\it finite memory} of type I in \cite{LW}
(see also  \cite{LW2} for more general relaxation kernels leading to uniform but not exponential decays).
In the case of memory of type II, the same \cite[Theorem 3.7]{brasilian} establishes the
exponential decay of the energy in the subcritical case, but under strong ``smallness" type restrictions
imposed on the mass of kernel $\varrho$. Here, again, this is an extension to infinite memory of the results obtained in \cite{LW}.
In short, this ``smallness"  condition requires a rather fast decay of the kernel $g$ with respect
to the strictly positive value $\alpha\beta-\gamma$.
For exponentially decaying kernels of the form $g(t)=\varrho \delta e^{-\delta t }$
with $\delta>0$, this condition translates into
$ \varrho < \beta - \gamma  \alpha^{-1} $.

Regarding the negative result in the case of memory
of type II (conjectured in \cite{DLP}), the paper \cite{brasilian} evokes the lack of dissipativity
of the generator for larger values of  $\varrho$. One should note that dissipativity is a property of the
considered inner product and, alone, cannot prove the conjecture stated in \cite{DLP}, i.e.\ to disprove exponential stability.
In summary, the analysis carried out in \cite{brasilian}
is inconclusive with respect to the open
question under consideration.

Coming instead to the MGT equation with
memory of type II in the critical case, it is still unknown whether exponential stability could be achieved with suitably
calibrated relaxation kernel (fast decay and small mass).
Positive results are available in the literature (see \cite{brasilian,LW}), but only in the subcritical regime.

\subsection*{Notation}
We define the family of nested Hilbert spaces depending on a parameter $r\in \R$
$$
H^r=\D(A^\frac{r}{2}),\qquad\,\, \l u,v\r_r = \l A^{\frac{r}{2}} u , A^{\frac{r}{2}} v \r, \qquad\,\, \| u\|_r = \|A^{\frac{r}{2}} u\|.
$$
The index $r$ will be always omitted whenever zero. Along the paper,
the H\"older, Young and Poincar\'e inequalities will be tacitly used in several occasions.
The phase space of our problem is
$$
\H=H^1 \times H^1 \times H,
$$
endowed with the (Hilbert) product norm
$$
\|(u,v,w)\|_\H^2 = \|u\|_1^2 + \|v\|_1^2 + \|w\|^2.
$$

\section{Well-Posedness}

\noindent
The existence and uniqueness result for \eqref{VOLT} is ensured by the following theorems.

\begin{theorem}
\label{p1}
If the derivative $g'$ is bounded on bounded intervals, then
for every initial datum $U_0\in \H$, equation \eqref{VOLT} admits a unique weak solution
$$U=(u,\partial_t u,\partial_{tt}u)\in \C([0,T],\H)$$
on the interval $[0,T]$, for any $T>0$.
\end{theorem}

\begin{theorem}
\label{EXUN}
Assume the mass restriction $\varrho <\beta$. Then,
for every initial datum $U_0\in \H$, equation~\eqref{VOLT} admits a unique weak solution
$$U=(u,\partial_t u,\partial_{tt}u)\in \C([0,\infty),\H)$$
whose corresponding energy
$$
\F(t)= \|u(t)\|_1^2 + \|\partial_t u(t)\|_1^2 + \|\partial_{tt} u(t)\|^2
+ \int_0^t g(t-s)\|\partial_t u(t)-\partial_t u(s)\|_1^2 \d s
$$
satisfies the energy inequality
\begin{equation}
\label{EI}
\F(t) \leq K \F(0) \e^{\omega t},
\end{equation}
for some structural constants
$K,\omega>0$ and for all $t\geq0$.
\end{theorem}

As already mentioned in the Introduction,
the first Theorem \ref{p1} above, within a very mild assumption on the derivative $g'$
of the memory kernel, provides existence and uniqueness of solutions
in the space ${\H}$ \emph{without} imposing the usual restriction $\varrho<\beta$ on the size of the mass of $g$.
In which case, the solutions will be typically unbounded in time and exhibit a ``rough" asymptotic behavior as $t \to\infty$.
To the best of our knowledge, this is the  first well-posedness result obtained for 
the MGT equation with memory without  assuming ``smallness" restrictions imposed on the  relaxation kernel.

If instead we assume $\varrho<\beta$, then the second Theorem \ref{EXUN} provides
existence and uniqueness of solutions in $\H$  which enjoy an exponential-type growth at infinity.

In order to prove the theorems, we first show a well-posedness result in the more regular space
$$\hat\H=H^2\times H^1\times H,$$
by constructing solutions to a memoryless nonhomogeneous MGT equation
exploiting the so-called MacCamy trick (see e.g.\ \cite{pandolfi}).

\begin{lemma}
\label{l1}
For every initial datum $U_0\in\hat \H$ (and every $\varrho>0$), equation \eqref{VOLT} admits a unique weak solution
$$U=(u,\partial_t u,\partial_{tt}u)\in \C([0,T],\hat\H)$$
on the interval $[0,T]$, for any $T>0$.
\end{lemma}

\begin{proof}
Let $T>0$ be arbitrarily fixed, and let $R_\mu(t) $ denote the resolvent operator associated with the kernel
$$\mu(s)=-\frac1\beta g(s).$$
This means that $R_\mu$ solves the equation
$$
R_\mu(t) + \int_0^t \mu(t-s) R_\mu(s)\,\d s = \mu(t),\quad \forall t\geq 0.
$$
We now rewrite \eqref{VOLT} in the form
$$
A^{\frac12}\partial_t u(t)+\int_0^t \mu(t-s) A^{\frac12}\partial_t u(s)\d s=Y(t),$$
where
$$Y(t)=-\frac1\beta\big[A^{-\frac12}\partial_{ttt} u(t) + \alpha A^{-\frac12}\partial_{tt} u(t) + \gamma A^{\frac12}u(t)\big].
$$
If we knew in advance that $Y \in \C([0,T],H)$,
then the function $X(t)=A^{\frac12}\partial_t u(t)$, being the solution to the Volterra equation on $[0,T]$
$$
X(t) + \int_0^t \mu(t-s) X(s)\,\d s = Y(t),
$$
has the explicit representation
$$
X(t) = Y(t) -\int_0^t R_\mu(t-s) Y(s)\, \d s.
$$
Applying $\beta A^\frac12$ to both sides, we conclude that the function $U=(u,\partial_t u,\partial_{tt}u)$
satisfies the equation
\begin{equation}
\label{fixed}
\partial_{ttt} u + \alpha \partial_{tt} u + \beta A \partial_t u + \gamma Au=Q_U,
\end{equation}
having set
$$Q_U(t)=\int_0^t R_\mu(t-s) [ \partial_{ttt}u(s)+\alpha\partial_{tt}u(s)+\gamma Au(s)]\, \d s.
$$
Integrating by parts the first term above, we obtain
\begin{align*}
Q_U(t)&=\int_0^t R'_\mu(t-s)\partial_{tt}u(s)\, \d s-R_\mu(t)\partial_{tt} u(0)+R_\mu(0)\partial_{tt} u(t)\\
&\quad +\int_0^t R_\mu(t-s) [\alpha\partial_{tt}u(s)+\gamma Au(s)]\, \d s.
\end{align*}
At this point, we observe that $R_\mu(0)=\mu(0)<0$. Thus, calling $\hat\alpha=\alpha-R_\mu(0)>0$ and
\begin{align}
\label{Qhat}
\hat Q_U(t)&=\int_0^t R'_\mu(t-s)\partial_{tt}u(s)\, \d s-R_\mu(t)\partial_{tt} u(0)\\
\nonumber
&\quad +\int_0^t R_\mu(t-s) [\alpha\partial_{tt}u(s)+\gamma Au(s)]\, \d s,
\end{align}
equation \eqref{fixed} reads
\begin{equation}
\label{fixedp}
\partial_{ttt} u + \hat\alpha \partial_{tt} u + \beta A \partial_t u + \gamma Au=\hat Q_U.
\end{equation}
Introducing the three-component vector
$${\mathcal Q}_U=(0,0,\hat Q_U),$$
we can write \eqref{fixedp} in the abstract form
\begin{equation}
\label{fixedab}
\ddt U=\A U+{\mathcal Q}_U,
\end{equation}
where
$$\A U= (\partial_t u,\partial_{tt} u, -\hat\alpha \partial_{tt} u - \beta A \partial_{t} u
-\gamma A u).$$
It is known from \cite{KLM} that the
MGT equation without memory generates a strongly continuous semigroup $S(t)=\e^{\A t}$ on $\hat\H$.
Hence, its nonhomogeneous version \eqref{fixedab} driven by a generic forcing term ${\mathcal Q}\in L^1(0,T;\hat \H)$
admits a unique solution
$$U=(u,\partial_t u,\partial_{tt}u)\in \C([0,T],\hat\H).$$
Accordingly, if ${\mathcal Q}_U\in L^1(0,T;\hat \H)$, from
the variation-of-constant formula we end up with
$$
U(t) = S(t) U_0 + \int_0^t S(t-s) {\mathcal Q}_U(s) \d s.
$$
In order to finish the proof, we merely apply the Banach contraction principle, first
on the space
$${\mathcal X}=\big\{U\in  \C([0,T_0],\hat \H)\,:\, U(0)=U_0\big\},$$
with $T_0$ sufficiently small,
and then reiterated (due to the linearity) to the intervals $[ nT_0, (n+1) T_0 ]$
until $T$ is reached.
\end{proof}

The next step is extending the result of Lemma \ref{l1} to the whole space $\H$.
This will be easily accomplished by standard density arguments, once a priori estimates involving initial data belonging to $\H$ are established.
Here, different arguments are needed depending whether we are in the framework of Theorem \ref{p1} or of Theorem \ref{EXUN}.

\begin{proof}[Proof of Theorem \ref{p1}]
Let $T>0$ be arbitrarily fixed.
We start from equation \eqref{fixedp} of the previous proof, that is,
$$
\partial_{ttt} u + \hat\alpha \partial_{tt} u + \beta A \partial_t u + \gamma Au= \hat Q_U,
$$
whose solution $U$ (which is in fact the solution to the original equation)
exists in $\hat\H $ for initial data $U_0 \in \hat \H$.
Then, we multiply by $2\partial_{tt} u $ in $H$, and
we add to both sides the term $2m\langle u,\partial_{t} u\rangle_1$ for $m>0$ to be fixed later. We obtain
the differential equality
\begin{align*}
\ddt \LL_m + 2\hat{\alpha}
\|\partial_{tt} u\|^2
 =2\gamma \|\partial_t u \|_1^2 +2m \l u, \partial_t u\r_1  + 2\l\hat Q_U, \partial_{tt} u\r,
 \end{align*}
where we set
$$\LL_m(t) = \|\partial_{tt}u (t)\|^2+\beta \| \partial_t u(t)\|^2_1 + 2\gamma \l u(t), \partial_t u(t)\r_1  + m \|u(t)\|^2_1.
$$
It is readily seen that, up to choosing $m>0$ sufficiently large, there exist $\kappa_2>\kappa_1>0$ such that
\begin{equation}
\label{kappa}
\kappa_1 \| U(t)\|^2_{\H} \leq  \LL_m(t) \leq \kappa_2 \|U(t)\|^2_{\H}.
\end{equation}
Accordingly,
\begin{equation}
\label{bel}
\ddt\LL_m \leq C\LL_m + 2 \l\hat Q_U, \partial_{tt} u\r,
\end{equation}
for some $C>0$.
We now claim that the inequality
\begin{equation}
\label{Qterm}
\int_0^t \l\hat Q_U(s), \partial_{tt} u(s)\r\d s  \leq \frac14 \LL_m(t) + C_T\LL_m(0)   + C_T\int_0^t \LL_m(s)\d s
\end{equation}
holds for every $t\in[0,T]$. Here and till the end of the proof,
$C_T>0$
denotes a {\it generic} constant, independent of the initial data,
but depending on $T$.
Then, integrating \eqref{bel} on $[0,t]$, we end up with
$$\LL_m(t)\leq C_T\LL_m(0)+C_T\int_0^t \LL_m(s)\d s,\quad\forall t\in[0,T],
$$
and the standard Gronwall lemma together with  \eqref{kappa} entail
$$\|U(t)\|_{\H} \leq C_T\|U(0)\|_{\H}, \quad\forall t \in [0,T].$$
We are left to prove \eqref{Qterm}. Recalling \eqref{Qhat}, we limit ourselves to
show the more difficult estimate of the higher-order term, namely,
$${\mathfrak I}:=\int_0^t\int_0^s R_\mu(s-y)\l u(y),\partial_{tt}u(s)\r_1\d y\d s.
$$
We write
\begin{align*}
\int_0^s R_\mu(s-y)\l u(y),\partial_{tt}u(s)\r_1\d y
&=\frac{\d}{\d s}\bigg[\int_0^s R_\mu(s-y)\l u(y),\partial_{t}u(s)\r_1\d y\bigg]\\
&- R_{\mu}(0)\l u(s), \partial_t u(s)\r_1
-\int_0^s R_\mu'(s-y)\l u(y),\partial_{t}u(s)\r_1\d y.
\end{align*}
Note that, within the boundedness assumption on $g'$, we have that $R_\mu'$ (as well as $R_\mu$) is bounded on $[0,T]$.
Then, integrating on $[0,t]$ the identity above, we are led to
\begin{align*}
{\mathfrak I}&=\int_0^t R_\mu(t-s)\l u(s),\partial_{t}u(t)\r_1\d s
-R_\mu(0)\int_0^t \l u(s), \partial_t u(s)\r_1\d s\\
&\quad-\int_0^t\int_0^s R_\mu'(s-y)\l u(y),\partial_{t}u(s)\r_1\d y\d s.
\end{align*}
We finally estimate the three terms in the right-hand side as follows:
\begin{align*}
&\int_0^t R_\mu(t-s)\l u(s),\partial_{t}u(t)\r_1\d s
-R_\mu(0)\int_0^t \l u(s), \partial_t u(s)\r_1\d s\\
&\leq \varepsilon\|U(t)\|_\H^2+\frac{C_T}\varepsilon
\int_0^t \|U(s)\|_\H^2\d s,
\end{align*}
for any $\varepsilon>0$ small, and
\begin{align*}
-\int_0^t\int_0^s R_\mu'(s-y)\l u(y),\partial_{t}u(s)\r_1\d y\d s&
\leq C_T \int_0^t \|U(s)\|_\H \int_0^s \|U(y)\|_\H\d y \d s\\
&\leq C_T\bigg(\int_0^t \|U(s)\|_\H \d s\bigg)^2\\
&\leq C_T\int_0^t \|U(s)\|_\H^2\d s.
\end{align*}
Therefore,
$${\mathfrak I}\leq \varepsilon\|U(t)\|_\H^2+\frac{C_T}\varepsilon
\int_0^t \|U(s)\|_\H^2\d s.
$$
The remaining terms of $\int_0^t \l\hat Q_U(s), \partial_{tt} u(s)\r\d s$, as we said, are controlled in a similar (in fact easier) way,
and at the end one has to use \eqref{kappa}. Only at that point, one fixes $\varepsilon$ in order to get the desired coefficient $1/4$ (or smaller)
in front of $\LL_m(t)$.
This finishes the proof.
\end{proof}

\begin{proof}[Proof of Thorem \ref{EXUN}]
We only need to show the energy inequality \eqref{EI}.
To this aim, similarly to the proof of Theorem \ref{p1}, we take the product in $H$ of \eqref{VOLT} and $2\partial_{tt}u$,
and we add to both sides the term $2m\langle u,\partial_{t} u\rangle_1$ for $m>0$ to be fixed later.
This yields
\begin{align*}
& \ddt \big[\|\partial_{tt}u(t)\|^2+\beta\|\partial_{t}u(t)\|_1^2+m\|u(t)\|_1^2+2\gamma\langle u(t),\partial_t u(t)\rangle_1\big]\\
&\quad-2\int_0^t g(t-s) \l \partial_t u(s) , \partial_{tt}u(t) \r_1 \,\d s\\
\noalign{\vskip1mm}\nonumber
&=2\gamma\|\partial_{t}u(t)\|^2_1-2\alpha\|\partial_{tt}u(t)\|^2+2m\langle u(t),\partial_{t} u(t)\rangle_1\\
\noalign{\vskip3mm}\nonumber
&\leq2(\gamma+m)\F(t).
\end{align*}
Next, calling
$G(t)=\int_0^t g(s) \,\d s$,
we compute the integral in the left-hand side as
\begin{align*}
&-2\int_0^t g(t-s) \l \partial_t u(s) , \partial_{tt}u(t) \r_1 \,\d s\\
&=  \ddt \Big[\int_0^t g(t-s)\|\partial_t u(t)-\partial_t u(s)\|_1^2\, \d s - G(t) \|\partial_tu(t)\|_1^2\Big]\\
&\quad+ g(t) \|\partial_tu(t)\|_1^2- \int_0^t g'(t-s)\|\partial_t u(t)-\partial_t u(s)\|_1^2\, \d s.
\end{align*}
Setting
\begin{align*}
\E_m(t)&=\|\partial_{tt}u(t)\|^2+(\beta-G(t))\|\partial_{t}u(t)\|_1^2+m\|u(t)\|_1^2+2\gamma\langle u(t),\partial_t u(t)\rangle_1\\
&\quad+\int_0^t g(t-s)\|\partial_t u(t)-\partial_t u(s)\|_1^2\, \d s,
\end{align*}
since $g$ is nonnegative and nonincreasing we
arrive at the inequality
$$\ddt \E_m\leq 2(\gamma+m)\F.$$
Recalling that $G(t)\leq \varrho<\beta$, is then clear that, up to choosing $m$ large enough,
$$\kappa_1 \F(t) \leq \E_m(t) \leq \kappa_2 \F(t),$$
for some $\kappa_2> \kappa_1>0$.
The desired conclusion follows by
an application of the Gronwall lemma.
\end{proof}

\begin{remark}
If $A$ is a bounded operator, the conclusions of Theorem \ref{EXUN} are easily attained
removing the restriction $\varrho<\beta$.
\end{remark}

\begin{remark}
As a final comment, it is interesting to observe that the trick of multiplying both sides of the equation by $\langle u,\partial_{t} u\rangle_1$,
employed in the proofs above, allows to provide a two-line proof of the well-posedness of the strongly damped wave
equation with the ``wrong" sign of $Au$, namely,
$$\partial_{tt}u+A \partial_t u-A u=0,$$
which highlights the essential parabolicity of the original equation. This is not the case if one has a lower-order dissipation.
Indeed, the equation
$$\partial_{tt}u+A^\vartheta \partial_t u-A u=0$$
is ill-posed for $\vartheta<1$, as the real part of the spectrum of the associated linear operator is not bounded above.
\end{remark}

\section{The Case of the Exponential Kernel}

\noindent
We now dwell on the particular case of the exponential kernel
$$
g(s)= \varrho \delta \e^{-\delta s},
$$
with
$$\varrho \in (0,\beta)\and \delta>0.$$
For this choice,
equation \eqref{VOLT} reads
\begin{equation}
\label{VOLTEXP}
\partial_{ttt} u(t) + \alpha \partial_{tt} u(t) + \beta A \partial_t u(t)
 + \gamma Au(t)-\varrho\delta \int_0^t \e^{-\delta (t-s)} A \partial_t u(s)\d s=0.
\end{equation}
In the same spirit of \cite{DLP}, taking the sum $\partial_t \eqref{VOLTEXP} + \delta\eqref{VOLTEXP}$
we obtain the fourth-order equation
\begin{equation}
\label{4TH}
\partial_{tttt} u + (\alpha+\delta) \partial_{ttt} u
+ \alpha\delta \partial_{tt} u + \beta A \partial_{tt} u + (\gamma + \delta\beta-\varrho\delta)A\partial_t u
 + \gamma\delta Au=0.
\end{equation}
Note that
$$
\gamma + \delta\beta-\varrho\delta>0,
$$
as $\varrho<\beta$.
Introducing the 4-component space
$$\mathcal{V}=H^1\times H^1\times H^1 \times H,$$
it is known from \cite{DPMILAN} that \eqref{4TH}
admits a unique (weak) solution
$$\hat U=(u,\partial_t u,\partial_{tt}u,\partial_{ttt}u)\in\C([0,\infty),\mathcal{V}),$$
for every initial datum $\hat U_0\in \mathcal{V}$.
Besides, the analysis in \cite{DPMILAN} provides necessary and sufficient conditions
in order for \eqref{4TH} to be (exponentially) stable, depending on two \emph{stability numbers} $\varkappa$ and $\varpi$,
which in turn depend only on the (positive) structural constants of the equation.
For this particular case, the two stability numbers read
$$\varkappa= \frac{\alpha\beta-\gamma+\varrho\delta}{\alpha+\delta}>0 \and \varpi= \frac{\alpha\beta\delta^2-\gamma\delta^2-\alpha\varrho\delta^2}{\gamma + \delta\beta - \varrho\delta}.$$
In particular, if $\varrho\in(\beta-\tfrac{\gamma}{\alpha},\beta)$ and $\delta$ is large enough, then
$$
\varpi <-\lambda_1 \varkappa
$$
where $\lambda_1>0$ is the smallest element of the spectrum $\sigma(A)$ of the operator $A$.
In this regime, the results of \cite{DPMILAN} predict the existence of solutions growing exponentially fast,
which gives a clear indication that our energy $\F$ might blow up exponentially for some initial data.
At the same time, the equivalence between \eqref{VOLTEXP} and \eqref{4TH} is, at this stage, only formal. The next proposition
establishes such an equivalence in a rigorous way.

\begin{proposition}
\label{gancio}
Let $U_0=(u_0,v_0,w_0)\in H^1\times H^1\times H^1$ be an arbitrarily fixed vector satisfying
the further regularity assumption
$$\beta v_0 + \gamma u_0 \in H^2.$$
Then
the projection $U=(u,\partial_t u,\partial_{tt}u)$ onto the first three components of the solution
$\hat U=(u,\partial_t u,\partial_{tt}u,\partial_{ttt}u)$ to \eqref{4TH} with initial datum
$$(u_0,v_0,w_0,-\alpha w_0  - A(\beta v_0 + \gamma u_0))\in\mathcal{V}$$
is the unique solution to
\eqref{VOLTEXP} with initial datum $U_0$.
\end{proposition}

\begin{proof}
We introduce the auxiliary variable
$$
\phi(t)= \partial_{ttt} u(t) + \alpha \partial_{tt} u(t)
+ \beta A \partial_{t} u(t) + \gamma A u(t).
$$
Since $u$ solves \eqref{4TH}, the function $\phi$ fulfils the identity
$$
\partial_t \phi + \delta \phi - \varrho \delta A \partial_t u =0.
$$
A multiplication by $\e^{\delta t}$ yields
$$
\ddt [\e^{\delta t} \phi(t)] - \varrho \delta \e^{\delta t} A \partial_t u(t)=0.
$$
Noting that $\phi(0)=0$, an integration on $[0,t]$ leads at once to \eqref{VOLTEXP}.
\end{proof}

Still, this is not enough to conclude that $\F$ can grow exponentially fast, since one has to verify that
this occurs for a particular trajectory of \eqref{4TH}, with initial conditions complying with the assumptions above.

\section{Exponentially Growing Solutions}

\noindent
In this section, we state and prove the second main result of the paper. Namely, we produce an example of
memory kernel $g$ for which equation \eqref{VOLT} admits solutions with energy growing exponentially fast. To this end,
we consider the exponential kernel $g(s)= \varrho \delta \e^{-\delta s}$ of the previous section. For simplicity,
we also assume that the spectrum of the operator $A$ contains at least one eigenvalue $\lambda>0$, which is
always the case in concrete situations.

\begin{theorem}
\label{EXPGROW}
Let $\varrho\in(\beta-\tfrac{\gamma}{\alpha},\beta)$ be arbitrarily fixed. Then, for every $\delta>0$ sufficiently large, there exist
$\varepsilon>0$, an initial datum $U_0\in \H$ and a sequence $t_n\to \infty$
such that the energy $\F(t)$ associated to the solution
to \eqref{VOLTEXP} originating from $U_0$ satisfies the estimate
$$
\F(t_n) \geq \lambda \e^{\varepsilon t_n}, \quad \forall n \in \N.
$$
\end{theorem}

\begin{remark}
According to \cite{LW}, for this particular kernel exponential stability occurs in the subcritical case, within the following assumption:\
there exist $k\in(\frac\gamma\beta,\alpha)$ and $\theta>\frac{k}\delta$ such that
$$
\varrho \leq \Big( \beta -\frac\gamma{k}\Big)\min\Big\{1, \frac{2}{ k (2+\theta) }\Big\}.
$$
The reader will have no difficulties to check that the condition above implies that  $\varrho < \beta -\frac{\gamma}{\alpha}$, which contradicts
$\varrho\in(\beta-\tfrac{\gamma}{\alpha},\beta)$ assumed in Theorem \ref{EXPGROW}.
\end{remark}

In order to prove the theorem,
we introduce the fourth-order polynomial in the complex variable $\xi$
\begin{equation}
\label{pol}
\P(\xi)= \xi^4 + (\alpha+\delta)\xi^3 +(\alpha\delta + \beta\lambda) \xi^2 + \lambda(\gamma + \delta\beta-\varrho\delta) \xi
+\lambda \gamma\delta.
\end{equation}
Moreover, for all $x\geq0$, we set
\begin{equation}
\label{qqq}
q(x) = \sqrt{\frac{4 x^3 +3(\alpha+\delta)x^2 + 2(\alpha\delta + \beta\lambda)x
+ \lambda(\gamma + \delta\beta-\varrho\delta)}{\alpha+\delta + 4x}}>0.
\end{equation}
The next algebraic result will be crucial for our purposes.

\begin{lemma}
\label{keylemma}
Let $\varrho\in(\beta-\tfrac{\gamma}{\alpha},\beta)$ be arbitrarily fixed. Then, for every $\delta>0$ sufficiently large,
there exists $p>0$ such that the complex number
$$
\hat \xi = p + \i q(p)
$$
solves the equation $\P(\hat \xi)=0$.
\end{lemma}

\begin{proof}
For all $x\geq0$, by direct calculations we find the equalities
$$
\Im [\P(x+ \i q(x))] =0 \and
\Re[\P(x+ \i q(x))] = {f}(x),
$$
where
\begin{align*}
{f}(x)& =  x^4 + (\alpha + \delta)x^3 + (\alpha\delta + \beta\lambda)x^2 + (\gamma + \delta\beta - \varrho\delta)\lambda x
+\gamma\delta\lambda+q(x)^4\\\noalign{\vskip1mm}
&\quad - (6 x^2 + 3(\alpha+\delta)x + \alpha\delta + \beta\lambda)q(x)^2.
\end{align*}
By means of direct computations, with the aid of \eqref{qqq} and the assumption $\varrho>\beta-\tfrac{\gamma}{\alpha}$,
it is readily seen that
$$\lim_{x\to+\infty }f(x)=-\infty$$
and
$$
{f}(0)= \gamma\delta\lambda+q(0)^4 -(\alpha\delta + \beta\lambda)q(0)^2 \sim \delta \lambda (\gamma-\alpha\beta + \alpha\varrho) >0,\quad \text{as } \delta\to+\infty.
$$
As a consequence,
being $f$ continuous on $[0,\infty)$, once $\delta>0$ has been fixed sufficiently large there exists
$p>0$ such that
$$
0=f(p)=\Re[\P(p + \i q(p))].$$
The proof is finished.
\end{proof}

\begin{proof}[Proof of Theorem \ref{EXPGROW}]
Denoting by $w \in H$ the normalized eigenvector of $A$ corresponding to $\lambda$,
we consider the function
$$
u(t)= \e^{pt}\, [ r \sin (q t) + \cos (q t) ]\,w.
$$
Here, $p>0$ is given by Lemma \ref{keylemma}, $q=q(p)>0$ is given by \eqref{qqq} and
$$
r=r(p)= \frac{p^3 -3pq^2+\alpha(p^2-q^2) +\beta\lambda p + \gamma\lambda}{q^3 - 3p^2 q -2\alpha p q - \beta\lambda q}.
$$
Note that $r$ is well defined, since \eqref{S1} and \eqref{qqq} ensure that
$$
q^3 - 3p^2 q -2\alpha p q - \beta\lambda q
= \frac{-q(8p^3 + 8\alpha p^2 + 2\alpha^2p + 2\beta\lambda p + \lambda \varrho\delta + \lambda(\alpha\beta-\gamma))}{\alpha+\delta+4p} <0.
$$
The function $u$ defined above solves the fourth-order equation \eqref{4TH}.
Indeed, calling for simplicity
$$\psi(t)=\e^{pt}\, [ r \sin (q t) + \cos (q t) ]$$
and recalling that due to Lemma \ref{keylemma} the complex number $p+\i q$ is a root of the polynomial $\P$ defined in~\eqref{pol},
we have
\begin{align*}
&\partial_{tttt} u + (\alpha+\delta) \partial_{ttt} u
+ \alpha\delta \partial_{tt} u + \beta A \partial_{tt} u + (\gamma + \delta\beta-\varrho\delta)A\partial_t u
 + \gamma\delta Au\\
&\quad = \Big[{\frac{\d^4 \psi}{\d t^4}} + (\alpha+\delta){\frac{\d^3 \psi}{\d t^3}}
+ (\alpha\delta + \beta\lambda){\frac{\d^2 \psi}{\d t^2}}
+ \lambda(\gamma + \delta\beta-\varrho\delta){\frac{\d \psi}{\d t}} +\lambda\gamma\delta\psi \Big]w\\\noalign{\vskip1.5mm}
&\quad = 0.
\end{align*}
Moreover, being
\begin{align*}
u(0) &= w,\\
\partial_t u(0) &= (p+rq)w,\\
\partial_{tt} u(0) &= (p^2+2rpq-q^2) w,\\
\partial_{ttt} u(0) &= (p^3 + 3 r p^2 q - 3 p q^2 -rq^3) w,
\end{align*}
thanks to the choice of $r$ it is true that
\begin{align*}
\partial_{ttt} u(0) &= -\alpha(p^2+2rpq-q^2)w - \beta \lambda(p+rq) w - \gamma \lambda w\\
& = -\alpha \partial_{tt} u(0)  - \beta A \partial_t u(0) -\gamma Au(0).
\end{align*}
Invoking Lemma~\ref{gancio}, the function $u$ turns out to be
the unique solution to
\eqref{VOLTEXP} corresponding to the initial datum
$$z_0 = (w,(p+rq) w,(p^2+2rpq-q^2) w).$$
Finally, setting
$$
t_n= \frac{2n\pi}{q}\to+\infty \and \varepsilon = 2 p>0,
$$
we conclude that
$$
\F(t_n) \geq \|u(t_n)\|_1^2 = \lambda \e^{\varepsilon t_n}.
$$
The proof of Theorem \ref{EXPGROW} is finished.
\end{proof}

\begin{Acknowledgments}
The authors would like to thank Monica Conti for fruitful discussion.
\end{Acknowledgments}


\end{document}